\documentclass[12pt]{amsart}
\usepackage{amsmath, amsthm, amssymb}
\usepackage[initials]{amsrefs}

\title{On VC-density over indiscernible sequences}
\author{Vincent Guingona,\\ Cameron Donnay Hill}
\address[Guingona]{University of Notre Dame \\ Department of Mathematics \\ 255 Hurley, Notre Dame, IN 46556}
\email{guingona.1@nd.edu}
\urladdr{http://www.nd.edu/~vguingon/}
\address[Hill]{Wesleyan University \\ Department of Mathematics and Computer Science \\  45 Wyllys Avenue, Middletown, CT 06459}
\email{cdhill@wesleyan.edu}
\date{\today}

\newtheorem{thm}{Theorem}[section]
\newtheorem{cor}[thm]{Corollary}
\newtheorem{lem}[thm]{Lemma}
\newtheorem{prop}[thm]{Proposition}

\newtheorem{ques}[thm]{Open Question}

\theoremstyle{remark}

\newtheorem{expl}[thm]{Example}

\theoremstyle{definition}
\newtheorem{defn}[thm]{Definition}

\newcommand{\tp}{\mathrm{tp} }

\newcommand{\comp}{\mathrm{compl} }

\renewcommand{\phi}{\varphi}
\newcommand{\CC}{\mathfrak{C}}

\renewcommand{\models}{\vDash}

\begin{document}

\begin{abstract}
 In this paper, we study VC-density over indiscernible sequences (denoted VC${}_{\text{ind}}$-density).  We answer an open question in \cite{ADHMS}, showing that VC${}_{\text{ind}}$-density is always integer valued.  We also show that VC${}_{\text{ind}}$-density and dp-rank coincide in the natural way.
\end{abstract}

\maketitle

\section{Introduction}

In recent years,  the examination of NIP (or dependent) theories has been especially fruitful. NIP includes many theories of general mathematical interest, such as o-minimal theories (including the theory of the real numbers) and C-minimal theories (including algebraically-closed valued fields). While there have been several remarkable successes, no one has yet provided a decisively ``correct'' notion of super-dependence in analogy with super-stability. However, one viable candidate notion is that of strong-dependence addressed by Shelah in \cite{Shelah715}, \cite{Shelah783}, \cite{Shelah863} and \cite{Shelah900}. Associated with the notion of strong-dependence is the notion of dp-rank, one of the subjects of this article. 
In some sense, dp-rank is indeed a measure of the complexity of families of definable sets relative to a theory.  In a recent paper of Kaplan, Onshuus, and Usvyatsov \cite{KOUdpmin}, they show that dp-rank is subadditive in the way one would expect a dimension to be (see Theorem \ref{Thm_Subadditivity} below, which is Theorem 4.8 of \cite{KOUdpmin}).

Following a somewhat different line of investigation, Aschenbrenner, Dolich, Haskell, MacPherson, and Starchenko in \cite{ADHMS} suggest using VC-density as a dimension for strongly-dependent theories.  In fact, the thesis of \cite{ADHMS} holds that VC-density is, in some sense, a more ``accurate'' measure of the complexity of a formula than the competitors, possibly including dp-rank. We say ``possibly,'' for the relationship between dp-rank and VC-density is not well understood. (It is known that VC-density bounds dp-rank from above -- see Proposition \ref{Prop_VCDPRelation} below -- but it is not clear that VC-density also provides a lower bound for dp-rank.)

This paper amounts to a first attempt to establish that relationship. In order to make headway, we simplify the situation by analyzing VC${}_{\text{ind}}$-density -- VC-density evaluated exclusively relative to indiscernible sequences -- instead of the full VC-density. In this restricted context, we answer several of the open questions around the relation between dp-rank and VC-density. Along the way, we establish some additional interesting results; for example, we show that VC${}_{\text{ind}}$-density is always integer-valued (Theorem \ref{Thm_VCindDensityIntegerValued} below).

The structure of this article is as follows. In the remainder of this section, we provide most of the important definitions and state the main results. In Section \ref{Sec_LocalDPRank}, we introduce a notion of ``local'' dp-rank, and we show that it is, indeed, the ``correct'' localization of the standard dp-rank. Though its theory is not deep in itself (or as relates to dp-rank), the local dp-rank is rather interesting in its relation to VC${}_{\text{ind}}$-density. Indeed, in Section \ref{Sec_VCindDP}, we show that local dp-rank and VC${}_{\text{ind}}$-density coincide. We derive Theorems \ref{Thm_VCindDensityIntegerValued} and  \ref{Thm_VCindDensityDPRank} as straightforward corollaries.  Finally, in Section \ref{Sec_Future}, we briefly discuss how the techniques employed in this paper to understand VC${}_{\text{ind}}$-density might be applied to the original questions surrounding VC-density.

\subsection{VC-density and VC${}_{\text{ind}}$-density}

Let $T$ be a complete first-order theory in some language $L$ and let $\CC$ be a monster model for $T$.

Let $p(x)$ be a partial type and let $\varphi(x;y)$ be any formula.  For any set $B \subseteq \CC_y$, let $S_\varphi(B) \cap [p]$ denote the set of all $\varphi$-types over $B$ consistent with $p(x)$.  That is, all maximal subsets of
\[
 \{ \varphi(x; b) : b \in B \} \cup \{ \neg \varphi(x; b) : b \in B \}
\]
consistent with $p(x)$.

\begin{defn}\label{Defn_VCDensity}
 Fix $\ell \in \mathbb{R}$.  We say that a formula $\varphi(x;y)$ has \emph{VC-density $\le \ell$ with respect to $p$} if there exists $K \in \mathbb{R}$ such that, for all finite $B \subseteq \CC_y$
 \[
  \left| S_\varphi(B) \cap [p] \right| \le K \cdot |B|^\ell.
 \]
 If $p(x)$ is the partial type $x = x$, we drop ``with respect to $p$.''
\end{defn}

One interesting open question, originally posed in \cite{ADHMS}, is can we determine the VC-density of formulas in many variables knowing the VC-density of formulas in one variable?

\begin{ques}\label{Ques_VCDenAdd}
 If there exists $k < \omega$ such that, for all $\varphi(x;y)$ with $|x| = 1$, the VC-density of $\varphi(x;y)$ is $\le k$, then does there exists a function $f: \omega \rightarrow \omega$ such that, for all $\varphi(x;y)$, the VC-density of $\varphi$ is $\le f( |x| )$?
\end{ques}

Another question is can we relate dp-rank and VC-density in a natural way?

\begin{ques}\label{Ques_VCDenDPMin}
 Fix $n < \omega$.  Is it true that $p(x)$ has dp-rank $\le n$ if and only if, for all formulas $\varphi(x;y)$, $\varphi$ has VC-density $\le n$ with respect to $p$?
\end{ques}

In particular, is it true that $T$ is dp-minimal if and only if the VC-density of any formula $\varphi(x;y)$ is $\le |x|$?  One direction is clear, namely:

\begin{prop}\label{Prop_VCDPRelation}
 Fix a partial type $p(x)$ and $n < \omega$.  If, for all formulas $\varphi(x; y)$, $\varphi$ has VC-density $\le n$ with respect to $p$, then $p$ has dp-rank $\le n$.
\end{prop}

This is a obvious generalization of the proof of Proposition 3.2 of \cite{DGL}.  These open questions are difficult to answer in this general setting.  However, if we restrict ourselves to indiscernible sequences, we can answer both of these questions.

\begin{defn}\label{Defn_VCindDensity}
 Fix $\ell \in \mathbb{R}$.  We say that a formula $\varphi(x;y)$ has \emph{VC${}_{\text{ind}}$-density $\le \ell$ with respect to $p$} if there exists $K \in \mathbb{R}$ such that, for all finite indiscernible sequences $\overline{b} = \langle b_i : i < N \rangle$, 
 \[
  \left| S_\varphi(B) \cap [p] \right| \le K \cdot N^\ell,
 \]
 where $B = \{ b_i : i < N \}$.
\end{defn}

In the remainder of this paper, we prove the following theorems.

\begin{thm}\label{Thm_VCindDensityIntegerValued}
 VC${}_{\text{ind}}$-density is integer valued.
\end{thm}

This answers an open question posed by Aschenbrenner, Dolich, Haskell, MacPherson, and Starchenko in \cite{ADHMS}.

\begin{thm}\label{Thm_VCindDensityDPRank}
 For any $n < \omega$, a partial type $p(x)$ has dp-rank $\le n$ if and only if all formulas $\varphi(x;y)$ have VC${}_{\text{ind}}$-density $\le n$ with respect to $p$.
\end{thm}

This answers Open Question \ref{Ques_VCDenDPMin} for VC${}_{\text{ind}}$-density.  As a corollary of Theorem \ref{Thm_VCindDensityDPRank} and Theorem 4.8 of \cite{KOUdpmin}, we also answer Open Question \ref{Ques_VCDenAdd} for VC${}_{\text{ind}}$-density.

\begin{cor}\label{Cor_VCindDenAdd}
 Fix $k < \omega$ and suppose that, for all $\varphi(x;y)$ with $|x| = 1$, the VC${}_{\text{ind}}$-density of $\varphi(x;y)$ is $\le k$.  Then, for all $\varphi(x;y)$, the VC${}_{\text{ind}}$-density of $\varphi$ is $\le k \cdot |x|$.
\end{cor}

\section{Local dp-rank}\label{Sec_LocalDPRank}

In order to prove Theorem \ref{Thm_VCindDensityIntegerValued} and Theorem \ref{Thm_VCindDensityDPRank}, we introduce a notion of local dp-rank, which is just the obvious localization of dp-rank.  This will turn out to be exactly equal to VC${}_{\text{ind}}$-density.

In this section, we borrow notation from Chapters 3 and 4 of \cite{SimonBook}.  Let $( I; < )$ be a linear order and let $\comp(I)$ denote the completion of $I$.  For any $C \subseteq \comp(I)$, let $\sim_C$ be the equivalence relation on $I$ defined as follows.  For $i, j \in I$, $i \sim_C j$ if and only if, for all $c \in C$, we have that $c \le i \Leftrightarrow c \le j$ and $i \le c \Leftrightarrow j \le c$.  Notice that, for any $C$, this is a convex equivalence relation on $(I; <)$.

\begin{defn}\label{Defn_dpRank}
 Fix a partial type $p(x)$ and $n < \omega$.  We say that $p$ has \emph{dp-rank} $\le n$ if, for all $a \models p$ and all indiscernible sequences $\overline{b} = \langle b_i : i \in \mathbb{Q} \rangle$, there exists $C \subseteq \mathbb{R}$ with $|C| \le n$ such that, for all $i, j \in \mathbb{Q}$ with $i \sim_C j$, $\tp(b_i / a) = \tp(b_j / a)$.
\end{defn}

This definition is equivalent to the standard definition of dp-rank using ICT-patterns by Proposition 4.20 of \cite{SimonBook}, which is a simple generalization of of Lemma 1.4 of \cite{Simon}.  This is also equivalent to the definition of dp-rank using mutually indiscernible sequences, as in \cite{KOUdpmin}.  We turn now to a localization of this definition.

Let $p(x)$ be any partial type and let $\varphi(x;y)$ be any formula.

\begin{defn}\label{Defn_LocalDPRank}
 For some $n < \omega$, we say that the \emph{(local) dp-rank of $\varphi$ with respect to $p$} is $\le n$ if, for all $a \models p$ and all indiscernible sequences $\overline{b} = \langle b_i : i \in \mathbb{Q} \rangle$ in $\CC_y$, there exists $C \subseteq \mathbb{R}$ with $|C| \le n$ such that, $i, j \in \mathbb{Q}$ with $i \sim_C j$, $\models \varphi(a; b_i) \leftrightarrow \varphi(a; b_j)$.
\end{defn}

This local dp-rank is the correct localization of dp-rank in the following sense.

\begin{prop}\label{Prop_DPRankWorks}
 A partial type $p(x)$ has dp-rank $\le n$ if and only if, for all formulas $\varphi(x;y)$, $\varphi$ has local dp-rank $\le n$ with respect to $p$.
\end{prop}

\begin{proof}
 Suppose $p$ has dp-rank $\le n$.  Fix $\varphi(x;y)$, $a \models p$, and $\overline{b} = \langle b_i : i \in \mathbb{Q} \rangle$ indiscernible.  By definition, there exists $C \subseteq \mathbb{R}$ with $|C| \le n$ so that, for all $i, j \in \mathbb{Q}$ with $i \sim_C j$, $\tp(b_i / a) = \tp(b_j / a)$.  In particular, $\models \varphi(a; b_i) \leftrightarrow \varphi(a; b_j)$.
 
 Suppose that, for all formulas $\varphi(x;y)$, $\varphi$ has local dp-rank $\le n$ with respect to $p$.  Fix $a \models p$, and $\overline{b} = \langle b_i : i \in \mathbb{Q} \rangle$ indiscernible.  Then, for each $\varphi(x;y)$, there exists a $C_\varphi \subseteq \mathbb{R}$ with $|C_\varphi| \le n$ such that, $i, j \in \mathbb{Q}$ with $i \sim_{C_\varphi} j$, $\models \varphi(a; b_i) \leftrightarrow \varphi(a; b_j)$.  Suppose that $C_\varphi$ is minimal such, and it is easy to see that such a set is unique.  Let $C = \bigcup_{\varphi(x;y)} C_\varphi$.  Clearly, for all $i, j \in \mathbb{Q}$ with $i \sim_C j$, we have $\tp(b_i / a) = \tp(b_j / a)$.  Therefore, it suffices to show that $|C| \le n$.  By means of contradiction, suppose $|C| > n$.  In particular, there exists $\varphi_0(x;y), ..., \varphi_n(x;y)$ so that
 \[
  \left| \bigcup_{k \le n} C_{\varphi_k} \right| > n.
 \]
 Let $\psi(x; y_1, y_2) = \bigvee_{k \le n} [ \varphi_k(x; y_1) \leftrightarrow \neg \varphi_k(x; y_2) ]$.  There exists a subsequence $\overline{b}'$ of pairs of $\overline{b}$, indexed by $(n+1) \times \mathbb{Z}$ ordered lexicographically, so that for $k \le n$ and $m \in \mathbb{Z}$, $\models \psi(a; b'_{k,m})$ if and only if $m = 0$.  Using compactness, this shows that $\psi$ has local dp-rank $> n$ with respect to $p$, contrary to assumption.
\end{proof}

We say that a theory is \emph{strongly dependent} if, for all partial types $p(x)$, the dp-rank of $p$ is finite (see \cite{Shelah863} or Section 4.3 of \cite{SimonBook}).  There are theories that are dependent but not strongly dependent.  However, a formula $\varphi(x;y)$ is dependent if and only if it has finite dp-rank with respect to $x=x$.  We use an example to illustrate this difference between local and global definitions.

\begin{expl}\label{Expl_InfiniteDPRank}
 Consider the language $L = \{ E_i : i < \omega \}$ with countably many binary relations and let $T$ be the $L$-theory where each $E_i$ is an equivalence relation and the intersection of any number of classes is infinite.  Clearly this theory is dependent (in fact, stable), but it is not strongly dependent.  However, any formula has finite local dp-rank.  The issue is that there are formulas with arbitrarily large local dp-rank; for example, consider
 \[
  \varphi_n(x; y_0, y_1, ..., y_{n-1}) = \bigwedge_{i < n} E_i(x; y_i).
 \]
\end{expl}

\section{Local dp-rank and VC${}_{\text{ind}}$-density coincide}\label{Sec_VCindDP}

Before we show that local dp-rank and VC${}_{\text{ind}}$-density coincide, we introduce another rank which is an indiscernible version of UDTFS-rank (see, for example, \cite{Guingona2}, \cite{GuinLas}, or \cite{ADHMS}).  Fix $p(x)$ a partial type and $\varphi(x;y)$ a formula.  For some $B \subseteq \CC_y$, $q(x) \in S_\varphi(B) \cap [p]$, and formula $\psi(y)$, we say that $\psi$ \emph{defines} $q(x)$ if, for all $b \in B$, $\models \psi(b)$ if and only if $\varphi(x; b) \in q(x)$.

\begin{defn}\label{Defn_UDTFSindRank}
 For $n < \omega$, we say that $\varphi$ has \emph{UDTFS${}_{\text{ind}}$-rank $\le n$ with respect to $p$} if there exists a finite set of formulas
 \[
  \Psi = \{ \psi_r(y; z_{1,0}, ..., z_{1,\ell}, ..., z_{n,0}, ..., z_{n,\ell}, w_1, ..., w_k) : r < R \}
 \]
 such that, for all finite indiscernible sequences $\overline{b} = \langle b_i : i < N \rangle$, there exists $j_1, ..., j_k < N$ such that, for all $q(x) \in S_\varphi(\overline{b}) \cap [p]$, there exists $i_1, ..., i_n < N$ and $r < R$ so that
 \[
  \psi_r(y; b_{i_1}, ..., b_{i_1 + \ell}, ..., b_{i_n}, ..., b_{i_n+\ell}, b_{j_1}, ..., b_{j_k}) \text{ defines } q(x).
 \]
\end{defn}

\begin{lem}\label{Lem_EasyDirection}
 If $\varphi$ has UDTFS${}_{\text{ind}}$-rank $\le n$ with respect to $p$, then it has VC${}_{\text{ind}}$-density $\le n$ with respect to $p$.
\end{lem}

\begin{proof}
 Under the hypothesis, a simple count reveals that $| S_\varphi(B) \cap [p] | \le L \cdot |B|^n$ (where $B = \{ b_i : i < N \}$).
\end{proof}

We now show that local dp-rank, VC${}_{\text{ind}}$-density, and UDTFS${}_{\text{ind}}$-rank coincide.

\begin{thm}\label{Thm_RanksCoincide}
 For a formula $\varphi(x;y)$, a partial type $p(x)$, and $n < \omega$, the following are equivalent:
 \begin{enumerate}
  \item $\varphi$ has local dp-rank $\le n$ with respect to $p$.
  \item $\varphi$ has UDTFS${}_{\text{ind}}$-rank $\le n$ with respect to $p$.
  \item $\varphi$ has VC${}_{\text{ind}}$-density $\le n$ with respect to $p$.
  \item $\varphi$ has VC${}_{\text{ind}}$-density $\le \ell$ with respect to $p$ for some $\ell \in \mathbb{R}$ with $n \le \ell < n+1$.
 \end{enumerate}
\end{thm}

\begin{proof}
 (1) $\Rightarrow$ (2): Suppose $\varphi$ has local dp-rank $\le n$ with respect to $p$.
 
 \vspace{5pt}
 
 \noindent \textbf{Claim.} There exists $\ell < \omega$ such that, for any finite indiscernible sequence $\overline{b} = \langle b_i : i < N \rangle$ and all $a \models p$, there exists $i_1 < ... < i_n < N$ such that, for all $k \le n$ and all $j, j'$ with $i_k + \ell < j < j' < i_{k+1}$, $\models \varphi(a; b_j) \leftrightarrow \varphi(a; b_{j'})$ (for $k = 0, n$, let $i_0+\ell = -1$ and $i_n = N$).
 
 \begin{proof}[Proof of Claim.]
  Suppose not.  Then, for each $\ell < \omega$, we can find some $\overline{b} = \langle b_i : i < N \rangle$ and $a \models p$ for which no such $i_1 < ... < i_n < N$ exist.  A simple compactness argument then shows that we can build an indiscernible sequence $\overline{b} = \langle b_i : i \in \mathbb{Q} \rangle$ and $a \models p$ which witnesses the fact that $\varphi$ has dp-rank $> n$ with respect to $p$.
 \end{proof}
 
 From the claim it is easy to explicitly show $\varphi$ has UDTFS${}_{\text{ind}}$-rank $\le n$ with respect to $p$.  As in the proof of Theorem 1.2 (ii) of \cite{Guingona2}, there exists finitely many formulas $\theta_s(y_1, y_2; w_1, ..., w_k)$ for $s < S$ such that, for any finite indiscernible sequence $\overline{b}$, either $\overline{b}$ is a indiscernible set or there exists $j_1, ..., j_k$ and $s < S$ so that $\theta_s(y_1, y_2; b_{j_1}, ..., b_{j_k})$ defines the indiscernible ordering of $\overline{b}$ (i.e., $\langle b_i, b_{i'} \rangle$ holds of this formula if and only if $i < i'$).
 
 We need now a formula for each possible configuration of truth values between a finite indiscernible sequence $\overline{b}$ and $a \models p$.  Consider the set
 \[
  X = (\{ 0 \} \times \{ 0, ..., n \}) \cup (\{ 1, ..., n \} \times \{ 0, ..., \ell \})
 \]
 and, for each $f : X \rightarrow \{ 0, 1 \}$ and $s < S$, define the formula
 \[
  \psi_{f,s}(y; z_{1,0}, ..., z_{1,\ell}, ..., z_{n,0}, ..., z_{n,\ell}, w_1, ..., w_k)
 \]
 which satisfies:
 \begin{itemize}
  \item [(i)] If $y = z_{i,j}$, then $\psi_{f,s}$ holds iff. $f(i,j) = 1$,
  \item [(ii)] If $y$ falls in between $z_{i,\ell}$ and $z_{i+1,0}$ via the ordering given by $\theta_s(y_1, y_2; w_1, ..., w_k)$, then $\psi_{f,s}$ holds iff. $f(0,i) = 1$.
 \end{itemize}
 These formulas take care of the case where we are dealing with an indiscernible sequence that is not an indiscernible set.  For the other case, we define formulas $\psi^*_f(y; z_1, ..., z_{n\ell})$ for each $f : \{ 0, 1, ..., n \ell \} \rightarrow \{ 0, 1 \}$ so that
 \begin{itemize}
  \item [(i)] For $i > 0$, if $y = z_i$, $\psi^*_f$ holds iff. $f(i) = 1$,
  \item [(ii)] For $y \neq z_i$ for all $i > 0$, $\psi^*_f$ holds iff. $f(0) = 1$.
 \end{itemize}
 Now, for any finite indiscernible sequence $\overline{b}$ and $a \models p$, let $i_1 < ... < i_n < N$ be given as in the claim.  If $\overline{b}$ is not an indiscernible set, take $j_1, ..., j_k < N$ and $s < S$ so that $\theta_s(y_1, y_2; b_{j_1}, ..., b_{j_k})$ defines the indiscernible ordering.  Choosing the appropriate $f : X \rightarrow \{ 0, 1 \}$ will show us that
 \[
  \psi_{f,s}(y; b_{i_1}, ..., b_{i_1 + \ell}, ..., b_{i_n}, ..., b_{i_n+\ell}, b_{j_1}, ..., b_{j_k}) \text{ defines } \tp_\varphi(a/B),
 \]
 as desired.  In the case where $\overline{b}$ is an indiscernible set, we get a similar result with $\psi^*_f$.
 
 (2) $\Rightarrow$ (3): Follows immediately from Lemma \ref{Lem_EasyDirection}.
 
 (3) $\Rightarrow$ (4): Trivial.
 
 (4) $\Rightarrow$ (1): Suppose the local dp-rank of $\varphi(x;y)$ is $> n$, witnessed by $a$ and $\langle b_i : i \in \mathbb{Q} \rangle$.  Choose $C \subseteq \mathbb{R}$ minimal so that, for all $i, j \in \mathbb{Q}$ with $i \sim_C j$, $\models \varphi(a; b_i) \leftrightarrow \varphi(a; b_j)$, so $|C| > n$.  If $C$ is infinite, $\varphi(x;y)$ has IP with respect to $p$, hence $\varphi(x;y)$ has infinite VC${}_{\text{ind}}$-density, so we may assume $C$ is finite.  Therefore, there exists at least $n+2$ infinite $\sim_C$ classes, call them $D_0 < D_1 < ... < D_{n+1}$.  For each consecutive pair $D_\ell$ and $D_{\ell+1}$, either the truth value of $\varphi(a; b_i)$ differs or there exists some $j_\ell \in \mathbb{Q}$ with $D_\ell < j_\ell < D_{\ell+1}$ on which the truth value of $\varphi(a; b_{j_\ell})$ differs from both.  If the truth values on $D_\ell$ and $D_{\ell+1}$ differ, choose $j_\ell \in D_{\ell+1}$ arbitrarily.  Now fix $N < \omega$ and $\overline{i} = \langle i_0, ..., i_n \rangle$ an $(n+1)$-tuple from $N$ with $i_0 < ... < i_n$ and let $B_N = \{ b_i : i = 0, 1, ..., N-1 \}$.  Choose an order-preserving map $g_{\overline{i}} : N \rightarrow \mathbb{Q}$ so that
 \begin{itemize}
  \item [(i)] $g_{\overline{i}}(i_\ell) = j_\ell$, and
  \item [(ii)] If $i_{\ell-1} < i < i_\ell$, then $g_{\overline{i}}(i) \in D_\ell$.
 \end{itemize}
 Then consider the indiscernible subsequence $\langle b_{g(i)} : i < N \rangle$ (for $g = g_{\overline{i}}$).  By indiscernibility, there exists $a_{\overline{i}} \models p$ so that, for all $i = 0, ..., N-1$,
 \[
  \models \varphi(a_{\overline{i}}; b_i) \leftrightarrow \varphi(a; b_{g(i)}).
 \]
 However, by the choice of $D_\ell$ and $j_\ell$, each type $\tp_\varphi(a_{\overline{i}} / B_N)$ is distinct for different choices of $\overline{i}$.  Therefore,
 \[
  | S_\varphi( B_N ) \cap [p] | \ge \binom{N}{n+1},
 \]
 which is on the order of $N^{n+1}$.  Hence, the VC${}_{\text{ind}}$-density of $\varphi$ is $\ge n+1$.  Therefore, for each $\ell < n+1$, VC${}_{\text{ind}}$-density of $\varphi$ is $> \ell$.  This is what we aimed to prove.
\end{proof}

Notice that Theorem \ref{Thm_RanksCoincide} (3) $\Leftrightarrow$ (4) provides a proof of Theorem \ref{Thm_VCindDensityIntegerValued}.  Moreover, Theorem \ref{Thm_RanksCoincide} (1) $\Leftrightarrow$ (3) and Proposition \ref{Prop_DPRankWorks} immediately give a proof of Theorem \ref{Thm_VCindDensityDPRank}.  Consider the following theorem.

\begin{thm}[Theorem 4.8 of \cite{KOUdpmin}]\label{Thm_Subadditivity}
 Suppose $A$ is a set, $a_1, a_2$ are tuples, and $k_1, k_2 < \omega$.  If $\tp(a_i / A)$ has dp-rank $\le k_i$ for $i = 1,2$, then $\tp(a_1, a_2 / A)$ has dp-rank $\le k_1 + k_2$.
\end{thm}

This, coupled with Theorem \ref{Thm_VCindDensityDPRank} gives a proof of Corollary \ref{Cor_VCindDenAdd}.  As a result of this corollary, we get the following.

\begin{cor}\label{Cor_VCindDensitydpMin}
 A theory $T$ is dp-minimal if and only if all formulas $\varphi(x;y)$ have VC${}_{\text{ind}}$-density $\le |x|$.
\end{cor}

\section{Future Directions}\label{Sec_Future}

It is our hope that techniques developed in this paper for VC${}_{\text{ind}}$-density can me modified to answer the open questions about VC-density.

The definition of UDTFS${}_{\text{ind}}$-rank above differs drastically from the original definition of UDTFS-rank given in \cite{GuinLas} (related to the VC$d$ property in \cite{ADHMS}).  Fix a partial type $p(x)$ and a formula $\varphi(x;y)$.  We say that $\varphi$ has \emph{UDTFS-rank $\le n$ with respect to $p(x)$} if there exists finitely many formulas $\psi_r(y; z_1, ..., z_n)$ for $r < R$ such that, for every finite $B \subseteq \CC_y$ and $q(x) \in S_\varphi(B) \cap [p]$, there exists $b_1, ..., b_n \in B$ and $r < R$ such that
\[
 \psi_r(y; b_1, ..., b_n) \text{ defines } q(x).
\]
It is easy to show that the VC-density of $\varphi$ with respect to $p$ is bounded by the UDTFS-rank of $\varphi$ with respect to $p$ (by a simple counting argument).  Moreover, the following holds.

\begin{prop}[Theorem 3.13 of \cite{GuinLas}, Theorem 5.7 of \cite{ADHMS}]\label{Prop_UDTFSrank}
 If there exists $k < \omega$ such that, for all formulas $\varphi(x;y)$ with $|x| = 1$, the UDTFS-rank of $\varphi$ is $\le k$, then, for all formulas $\varphi(x;y)$, the UDTFS-rank of $\varphi$ is $\le k \cdot |x|$.
\end{prop}

However, it is easy to show that VC-density and UDTFS-rank do not coincide, so Proposition \ref{Prop_UDTFSrank} cannot be directly applied to answer Open Question \ref{Ques_VCDenAdd}.  Can we find a new rank by modifying UDTFS-rank to make it look similar to UDTFS${}_{\text{ind}}$-rank such that the following two conditions hold:
\begin{itemize}
 \item [(i)] this rank is subadditive as in Proposition \ref{Prop_UDTFSrank}, and
 \item [(ii)] this rank bounds VC-density and vice versa?
\end{itemize}
If such a rank exists, this would settle Open Question \ref{Ques_VCDenAdd}.  We cannot hope for a (integer valued) rank that coincides exactly with VC-density, since VC-densiy can be non-integer valued.  For example, see Proposition 4.6 of \cite{ADHMS}.

Another interesting question arising from our investigation is that of when VC-density and VC${}_{\text{ind}}$-density coincide. It is very easy to see that VC${}_{\text{ind}}$-density is bounded by VC-density, but it is not at all clear how one might use VC${}_{\text{ind}}$-density to estimate or bound the general VC-density. These two values certainly do not coincide in general, for VC${}_{\text{ind}}$-density is always integer-valued, while VC-density can take non-integer values even in very simple scenarios. Quite recently, however, Johnson (see Corollary 3.5 of \cite{Johnson}) has demonstrated a necessary and sufficient condition for the VC-density and VC${}_{\text{ind}}$-density of a given formula to be equal.

\end{document}